\documentclass[preprint,12pt]{elsarticle}
 \usepackage{amsmath}
 \usepackage{amsfonts}
 \usepackage{amssymb}
 \usepackage{amsthm}
 \usepackage{newlfont}
 \usepackage{lscape}
 \usepackage{epsfig,layout}
 \usepackage{color}
 \usepackage[normalem]{ulem}
 \usepackage{holtpolt}

 \theoremstyle{theorem}
 \newtheorem{thm}{Theorem}[section]
 \newtheorem{cor}[thm]{Corollary}
 
 \newtheorem{prop}[thm]{Proposition}
 \theoremstyle{definition}
 
 \theoremstyle{remark}
 \newtheorem{rem}[thm]{Remark}
 \newtheorem{exa}[thm]{Example}

 \journal{J. Math. Anal. Appl.}

\begin{document}

 \begin{frontmatter}

 \title{Self-adjointness of unbounded tridiagonal operators and spectra of their finite truncations}

 \author[1]{E. N. Petropoulou}
 \ead{jenpetr@upatras.gr}
 \address[1]{Department of Civil Engineering, Division of Geotechnical Engineering and Hydraulic Engineering, University of Patras, 26504 Patras, Greece.}

 \author[2]{L. Vel\'azquez\corref{3}}
 \ead{(velazque@unizar.es)}
 \address[2]{Department of Applied Mathematics $\&$ IUMA, Universidad de Zaragoza,
C/Mar\'{i}a de Luna 3, 50018 Zaragoza, Spain.}
 \cortext[3]{Corresponding author}

 \begin{abstract}
This paper addresses two different but related questions regarding an unbounded symmetric tridiagonal operator: its self-adjointness and the approximation of its spectrum by the eigenvalues of its finite truncations. The sufficient conditions given in both cases improve and generalize previously known results. It turns out that, not only self-adjointness helps to study limit points of eigenvalues of truncated operators, but the analysis of such limit points is a key help to prove self-adjointness. Several examples show the advantages of these new results compared with previous ones. Besides, an application to the theory of continued fractions is pointed out.
 \end{abstract}


 \end{frontmatter}

 \section{Introduction}
 \label{Intro}
 \setcounter{equation}{0}
 \renewcommand{\theequation}{\thesection.\arabic{equation}}
 \setcounter{thm}{0}
 \renewcommand{\thethm}{\thesection.\arabic{thm}}
 \par
Symmetric tridiagonal matrices provide the canonical matrix representations of self-adjoint operators in Hilbert spaces \cite{S1932} and, as a consequence, they naturally emerge in phenomena governed by self-adjoint operators. On the other hand, self-adjoint operators are ubiquitous in practical applications because of the usual requirement of a real spectrum in physical problems. Due to these reasons, symmetric tridiagonal operators appear in many areas of mathematics and physics.

A symmetric tridiagonal operator $T$ in an infinite dimensional Hilbert space $(H,(\cdot,\cdot))$ with an orthonormal base $\{e_n\}_{n=1}^\infty$ is given without loss by
 \begin{equation}
 Te_n = a_ne_{n+1} + b_ne_n + a_{n-1}e_{n-1},
 \quad a_n>0, \quad b_n\in\mathbb{R}, \quad n=1,2,\dots,
 \label{Intro_T}
 \end{equation}
where $e_0=0$. The matrix representation of $T$ in the basis $\{e_n\}_{n=1}^\infty$ is
 \begin{equation}
 J=\left( {\begin{array}{*{20}c}
   {b_1 } & {a _1 } & 0 & 0 & 0 & {...}  \\
   {a _1 } & {b_2 } & {a _2 } & 0 & 0 & {...}  \\
   0 & {a _2 } & {b_3 } & {a _3 } & 0 & {...}  \\
   0 & 0 & {a _3 } & {b_4 } & {a _4 } & {...}  \\
   {...} & {...} & {...} & {...} & {...} & {...}  \\
   \end{array}} \right),
 \label{Intro_Jacobi}
 \end{equation}
which is known as a Jacobi matrix. It is assumed that $a_n>0$ because the complex conjugated upper and lower diagonals can be made non-negative by a change of basis $e_n \to \eta_ne_n$, $|\eta_n|=1$, while setting $a_n=0$ for some $n$ splits \eqref{Intro_Jacobi} into a direct sum of Jacobi matrices that can be analyzed independently.

If $P_N$ is the orthogonal projection onto the subspace $H_N=\operatorname{span}\{e_n\}_{n=1}^N$, the composition $T_N=P_N T P_N$ defines an operator in $H_N$ called the orthogonal truncation of $T$ on $H_N$. Its matrix representation in the basis $\{e_n\}_{n=1}^N$ is the principal submatrix of \eqref{Intro_Jacobi} of order $N$.

Expression \eqref{Intro_T} defines a symmetric operator in the linear span of $\{e_n\}_{n=1}^\infty$, but we will identify $T$ with the closure of such an operator, which is known to be symmetric too. Then, either $T$ is self-adjoint, or $T$ has infinitely many self-adjoint extensions. In the latter case, the self-adjoint extensions have pure point spectra with any two disjoint \cite[Theorem 4.11]{S1998}. Different self-adjoint extensions can appear only when $T$ is unbounded, which is equivalent to saying that some of the sequences $a_n$ or $b_n$ is unbounded. Thus, self-adjointness is non trivial only in the unbounded case, which is also of practical interest since unbounded operators naturally appear in applications.

The general spectral problem for unbounded Jacobi matrices and, more specifically, approximation problems concerning such a spectrum have been considered in several studies. Indicatively we mention the recent works \cite{CJ2007,DP2002,IKP2007,JM2007,JN2001,JN2004,JNS2009,M2007,M2009,M2010,MZ2012,S2008,V2004}.

The present paper deals with two closely related problems concerning unbounded symmetric tridiagonal operators $T$: the search for self-adjointness conditions for $T$ which go further than known ones, and the possibility of approximating the spectrum $\sigma(T)$ of $T$ via the spectra $\sigma(T_N)$ of its orthogonal truncations $T_N$. To be more precise, let us denote by $\Lambda(T)$ the set of all limit points of the eigenvalues of $T_N$ when $N\to\infty$, i.e.
 \[
 \begin{gathered}
 \Lambda(T) = \left\{
 \lambda\in\displaystyle\mathop{\operatorname{Lim}}_{N\to\infty}\lambda_N :
 \lambda_N\in\sigma(T_N)
 \right\},
 \\
 \mathop{\operatorname{Lim}}_{N\to\infty}\lambda_N =
 \text{set of limit points of the sequence } \lambda_N.
 \end{gathered}
 \]
Information about $\Lambda(T)$ is of great importance not only from the point of view of operator theory, but also for the theory of continued fractions, orthogonal polynomials and numerical analysis (see \cite{IP2001} and the references therein).

In particular, the eigenvalues of $T_N$ coincide with the zeros of the polynomial $p_{N+1}(x)$ given by the recurrence relation
 \begin{equation}
 a_{n}p_{n+1}(x)+b_{n}p_{n}(x)+a_{n-1}p_{n-1}(x)=xp_{n}(x),
 \quad n=1,2,\ldots,
 \label{Intro_OPRR}
 \end{equation}
with $p_{0}(x)=0$ and $p_{1}(x)=1$. Thus, $\Lambda(T)$ coincides with the set of limit points of the zeros of the orthogonal polynomials $p_n(x)$ satisfying (\ref{Intro_OPRR}).

Besides, if $T$ is self-adjoint, the Jacobi continued fraction
 \begin{equation}
 K(\lambda) = \polter{1}{\lambda-b_1} - \polter{a_1^2}{\lambda-b_2} -
 \polter{a_2^2}{\lambda-b_3} - \cdots
 \label{Intro_CF}
 \end{equation}
converges to the function $\displaystyle\left((\lambda-T)^{-1}e_1,e_1\right)$ for every $\lambda\in\mathbb{C}\setminus\Lambda(T)$ \cite{B1994,IP2001}.

The self-adjointness of $T$ ensures the inclusion $\Lambda(T)\supseteq\sigma(T)$, although in general it does not guarantee the equality $\Lambda(T)=\sigma(T)$ (see for instance \cite{A1994,B1994,IP2001,IS1995,S1932}, and also \cite[Proposition 2.1]{CMV2006} for a generalization to normal band operators). When $T$ is not self-adjoint even the inclusion $\Lambda(T)\supseteq\sigma(T)$ can fail. This means that the relation between $\Lambda(T)$ and $\sigma(T)$ is more involved for an unbounded symmetric tridiagonal operator $T$ than for a bounded one.

In the bounded case several sufficient conditions for $\Lambda(T)=\sigma(T)$ can be found in the literature, but not much is known in the unbounded case (see \cite{IKP2007,IP2001}, also \cite{BLMT1998,BLT1995} for related problems concerning non-symmetric tridiagonal operators, and \cite{CMV2006} for extensions to unitary CMV operators). In particular, the authors of \cite{IKP2007} established sufficient conditions for $\Lambda(T)=\sigma(T)$ when $b_n$ is divergent, which regard the limits of some functions of $a_n$, $b_n$. However, the results as stated in \cite{IKP2007} only ensure that these conditions lead to $\Lambda(T)=\sigma(T)$ when $a_n$ is bounded. Otherwise they simply imply $\Lambda(T)\subseteq\sigma(T)$, while the opposite inclusion needs the additional assumption that $T$ is self-adjoint.

In this paper we push forward in different directions the ideas introduced in \cite{IKP2007} to study the unbounded case. In \S\ref{BR}, we start with a brief review of the results about $\Lambda(T)$ in \cite{IKP2007}, together with a new general result on self-adjointness which is achieved using limit point arguments (see Theorem \ref{thm-self} and Remark \ref{rem-interlace}). Then, the procedure used in \cite{IKP2007} is described so that it can be iterated to generate infinitely many sufficient conditions for $\Lambda(T)\subseteq\sigma(T)$. It is also proved that any of these conditions guarantees by itself the self-adjointness of $T$, thus the equality $\Lambda(T)=\sigma(T)$ (see Theorems \ref{Ifantis-update} and \ref{strong}).

This will prepare us to \S\ref{GmC} which discusses the recursion leading to the alluded infinitely many conditions for self-adjointness and $\Lambda(T)=\sigma(T)$. Although these conditions become exponentially intricate as the recursion advances, taking advantage of their qualitative dependence on $a_n$ and $b_n$ allows us to obtain very simple and general conditions for $\Lambda(T)=\sigma(T)$ which cover many of the examples in the literature (see Theorem \ref{weak}).

The iterative procedure giving infinitely many self-adjointness conditions can be exported to other contexts than the analysis of the set $\Lambda(T)$. In \S\ref{OmC}, we apply this idea to Carleman's criterion, and also to a self-adjointness condition which resembles another one due to J.~Janas and S.~Naboko.

In \S\ref{Comparisons}, several examples show the usefulness of the sufficient conditions previously obtained, comparing them with known results. Finally, some consequences in the theory of continued fractions are remarked in \S\ref{Applications}.

 \section{Basic results on $\Lambda(T)$ and self-adjointness}
 \label{BR}
 \setcounter{equation}{0}
 \renewcommand{\theequation}{\thesection.\arabic{equation}}
 \setcounter{thm}{0}
 \renewcommand{\thethm}{\thesection.\arabic{thm}}
 \par
Let $T$ be the operator defined by (\ref{Intro_T}), considered as the closure of that one with domain span$\{e_n\}_{n=1}^\infty$. The following results were proved in \cite{IKP2007}:
 \begin{align}
 \text{If $\lim_{n\to\infty}a_n=0$ then
 $T$ is self-adjoint and $\Lambda(T)=\sigma(T)$.}
 \label{IfantisThm1}
 \end{align}
 Also, if $\lim_{n\to\infty}b_n=\infty$, any of the conditions
 \begin{align}
 & \lim_{n\to\infty}\frac{a_n a_{n-1}}{b_n}=0,
 \label{IfantisThm2}
 \\
 & \lim_{n\to\infty} \frac{a_n a_{n-1}(a_{n-1}+a_{n-2})}{b_n b_{n-1}}=0,
 \label{IfantisThm3}
 \\
 & \lim_{n\to\infty}\frac{a_n a_{n-1}}{b_nb_{n-1}}
 \left[
 \frac{a_{n-1}^{2}}{b_n} + \frac{a_{n-2}(a_{n-2} + a_{n-3})}{b_{n-2}}
 \right]=0,
 \label{IfantisThm4}
 \end{align}
implies that $\Lambda(T)\subseteq\sigma(T)$. Furthermore, this inclusion becomes an equality when $T$ is self-adjoint. However, \cite{IKP2007} does not address the question of the self-adjointness of $T$ under conditions \eqref{IfantisThm2}--\eqref{IfantisThm4}, which is capital to guarantee the equality $\Lambda(T)=\sigma(T)$.

The proofs of the above results rely on a few arguments which we explicitly dissect below as a first step to carry out the extension of the method leading to \eqref{IfantisThm1}--\eqref{IfantisThm4}. Moreover, these arguments will also be used to prove that conditions \eqref{IfantisThm2}--\eqref{IfantisThm4} and their eventual extensions actually ensure the self-adjointness of $T$ and hence the equality $\Lambda(T)=\sigma(T)$.

The truncated operator $T_N$ is self-adjoint and has a complete set of orthonormal eigenvectors in $H_N$ with distinct real eigenvalues. Assuming $\lambda\in\Lambda(T)$ is equivalent to the existence of a subsequence of $T_N$, which will be also denoted by $T_N$ without loss, such that
 \begin{equation}
 \lim_{N\to\infty}\lambda_N=\lambda,
 \quad
 T_N x_N=\lambda_N x_N, \quad \|x_N\|=1, \quad x_N\in H_N.
 \label{Pre_BasicEigen}
 \end{equation}
The splitting
 $\|Tx_N\|^2 = \|P_NTx_N\|^2 + \|(I-P_N)Tx_N\|^2
 = \lambda_N^2 + a_N^2|(x_N,e_N)|^2$
gives the identity
 \begin{equation}
 \|(T-\lambda)x_N\|^2 = \|Tx_N\|^2 + \lambda^2 - 2\lambda\lambda_N
 = (\lambda-\lambda_N)^2 + a_N^2|(x_N,e_N)|^2.
 \label{IfantisId3bis}
 \end{equation}
As a consequence of this result, the condition
 \begin{equation}
 \lim_{N\to\infty} a_N (x_N,e_N) = 0
 \label{gencond}
 \end{equation}
implies that $\lim_{N\to\infty}\|(T-\lambda)x_N\|=0$, so that the limit point $\lambda$ lies on $\sigma(T)$.

The rest of the idea consists in finding asymptotic conditions for $a_n$ and $b_n$ ensuring \eqref{gencond} for any sequence $x_N$ of eigenvectors of $T_N$ with a   convergent sequence $\lambda_N$ of eigenvalues (actually, the only assumption in \cite{IKP2007} to obtain such conditions is the boundedness of $\lambda_N$). Bearing in mind the previous comments, these asymptotic conditions imply that $\Lambda(T)\subseteq\sigma(T)$.

The surprising new result is that condition \eqref{gencond} is also key to guarantee the self-adjointness of $T$ and thus the opposite inclusion $\Lambda(T) \supseteq \sigma(T)$. This result, missing in \cite{IKP2007} despite the close connection with the ideas developed there, will allow us to improve the consequences of \eqref{IfantisThm2}--\eqref{IfantisThm4} by ensuring the self-adjointness of $T$ and the equality $\Lambda(T)=\sigma(T)$ with no additional assumption.

\begin{thm} \label{thm-self}
Let $T_N$ be a subsequence of truncations of $T$. If there exists a sequence $x_N$ of normalized eigenvectors of $T_N$ with bounded eigenvalues and satisfying \eqref{gencond}, then $T$ is self-adjoint.
\end{thm}

\begin{proof}
Suppose that $x_N$ are normalized eigenvectors of $T_N$ with bounded eigenvalues $\lambda_N$. We can assume without loss that $\lambda_N$ converges to some point $\lambda$ by restricting to a new subsequence if necessary. Then, \eqref{IfantisId3bis} holds not only for $T$, but also for every extension of $T$. As a consequence, \eqref{gencond} implies that the limit point $\lambda$ lies in the spectrum of any such extension. In particular, if $T$ is not self-adjoint, $\lambda$ must be a common point of the spectra of the infinitely many self-adjoint extensions of $T$. This is in contradiction with the fact that any two self-adjoint extensions have disjoint spectra \cite[Theorem 4.11]{S1998}. Therefore, $T$ must be self-adjoint.
\end{proof}

\begin{rem} \label{rem-interlace}
It is known that the eigenvalues of $T_N$ always interlace with those of $T_{N+1}$. Even more, the bounded interval defined by any pair of eigenvalues of $T_N$ includes an eigenvalue of $T_n$ for any $n>N$ (see for instance \cite[Chapter 1]{C1978}). This shows that the existence of a subsequence $T_N$ having a sequence of normalized eigenvectors $x_N$ with bounded eigenvalues is guaranteed for any symmetric tridiagonal operator $T$. Therefore, every condition on $a_n$ and $b_n$ implying \eqref{gencond} for any such sequence $x_N$ gives simultaneously the inclusion $\Lambda(T)\subseteq\sigma(T)$ and the self-adjointness of $T$, leading to the equality $\Lambda(T)=\sigma(T)$.
\end{rem}

\begin{cor}
If $T$ is not self-adjoint, then
 \begin{equation}
 \liminf_{n\to\infty}
 \frac{a_n^2 p_n(\lambda_n)^2}{\sum_{k=1}^n p_k(\lambda_n)^2} > 0
 \label{gencond-OP}
 \end{equation}
for any bounded sequence $\lambda_n$ with $p_{n+1}(\lambda_n)=0$, where $p_n(x)$ are the orthogonal polynomials given by \eqref{Intro_OPRR}.
\end{cor}

\begin{proof}
First of all, it is known that the eigenvalues $\lambda_N$ of $T_N$ are the zeros of $p_{N+1}(x)$, with $\sum_{k=1}^N p_k(\lambda_N) e_k$ as eigenvectors, as follows directly from \eqref{Intro_OPRR} \cite{S1932}. If \eqref{gencond-OP} fails, there exists a bounded subsequence $\lambda_N$ of zeros of $p_{N+1}(x)$ such that \eqref{gencond} holds for $x_N=\big[\sum_{k=1}^N p_k(\lambda_N)^2\big]^{-1/2}\sum_{k=1}^N p_k(\lambda_N) e_k$. According to Theorem \ref{thm-self}, $T$ must be self-adjoint because $x_N$ are normalized eigenvectors of $T_N$ with bounded eigenvalues $\lambda_N$.
\end{proof}

Let us describe now the procedure to obtain \eqref{IfantisThm1}--\eqref{IfantisThm4} in such a way that it can be iterated to generate infinitely many other sufficient conditions for the inclusion $\Lambda(T)\subseteq\sigma(T)$. This will also help to prove that these conditions actually yield the equality $\Lambda(T)=\sigma(T)$ because they ensure that $T$ is self-adjoint.
The sketch of the referred procedure is as follows:

 \begin{itemize}

 \item
Write in coordinates the eigenvalue equation in (\ref{Pre_BasicEigen}), i.e.
 \begin{equation}
 \begin{aligned}
 & (\lambda_N-b_k)\delta_k=a_k\delta_{k+1}+a_{k-1}\delta_{k-1},
 \\
 & \delta_k=\delta_{N,k}=(x_N,e_k),
 \end{aligned}
 \qquad
 k=1,\ldots,N,
 \label{IfantisId2}
 \end{equation}
where we use the convention $\delta_0=\delta_{N+1}=0$.

 \item
Use the last $m$ equations of \eqref{IfantisId2} and $|\delta_k|\le\|x_N\|=1$ to find a bound for $\delta_N=(x_N,e_N)$ depending only on the eigenvalue $\lambda_N$ and the last $m$ coefficients $a_{N-k-1}$, $b_{N-k}$, $k=0,\dots,m-1$, of the truncation $T_N$, i.e.
 \begin{equation}
 |\delta_N| \le F_{m,N} =
 F_{m,N}(\lambda_N;b_N,a_{N-1},\dots,b_{N-m+1},a_{N-m}).
 \label{IfantisId4}
 \end{equation}

 \item
Give an asymptotic condition for $a_n$ and $b_n$ which ensures that
 \begin{equation}
 \lim_{N\to\infty}a_NF_{m,N}=0
 \label{IfantisId5}
 \end{equation}
when $\lambda_N$ is bounded.

 \end{itemize}

The asymptotic conditions for $a_n$ and $b_n$ found with the above procedure imply \eqref{gencond} for any sequence $x_N$ of eigenvectors of $T_N$ with bounded eigenvalues. Therefore, from Theorem \ref{thm-self} and Remark \ref{rem-interlace} we conclude that these conditions are sufficient not only for the inclusion $\Lambda(T)\subseteq\sigma(T)$, but also for the equality $\Lambda(T)=\sigma(T)$ and the self-adjointness of $T$.

The bound $F_{m,N}$ in (\ref{IfantisId4}) is deduced from the repetitive use of (\ref{IfantisId2}) for various values of $k$. Since the qualitative expression of $F_{m,N}$ will be needed later on, it is convenient to show the procedure leading to $F_{m,N}$ for the first values of $m$, introducing at the same time a notation which will make easier the transition to a general
index $m$. Denoting
 \begin{equation}
 c_n^- = c_n^-(\lambda_N) = \frac{a_{n-1}}{\lambda_N-b_n},
 \qquad
 c_n^+ = c_n^+(\lambda_N) = \frac{a_n}{\lambda_N-b_n},
 \label{c}
 \end{equation}
equations \eqref{IfantisId2}, ordered from the last to the first one, read as
 \begin{equation}
 \begin{aligned}
 & [N] & & \delta_N = c_N^-\delta_{N-1},
 \\
 & [N-1] & & \delta_{N-1} = c_{N-1}^-\delta_{N-2} + c_{N-1}^+\delta_N,
 \\
 & [N-2] & & \delta_{N-2} = c_{N-2}^-\delta_{N-3} + c_{N-2}^+\delta_{N-1},
 \\
 & & & \dots\dots\dots\dots\dots\dots\dots\dots\dots\dots
 \\
 & [N-k] & & \delta_{N-k} = c_{N-k}^-\delta_{N-k-1} + c_{N-k}^+\delta_{N-k+1},
 \\
 & & & \dots\dots\dots\dots\dots\dots\dots\dots\dots\dots
 \\
 & [1] & & \delta_1 = c_1^+\delta_{2}.
 \end{aligned}
 \label{IfantisId2bis}
 \end{equation}
Note that equation $[N-k]$ of \eqref{IfantisId2bis} requires $\lambda_N\ne b_{N-k}$. This will not be a problem because we will be interested in the limit $N\to\infty$ and we will only deal with the case $\lim_{n\to\infty}|b_n|=\infty$, which implies that $\lambda_N\ne b_{N-k}$ for $N$ big enough and fixed $k$ whenever $\lambda_N$ is bounded.

\smallskip

\noindent \underline{The bound $F_{0,N}$}

Using no equation of \eqref{IfantisId2bis} gives $F_{0,N}=1$, which leads to \eqref{IfantisThm1} due to Carleman's self-adjointness condition \eqref{CAR} \cite{C1923} (see also \cite[Chapter VII]{B1968}).

\smallskip

\noindent \underline{The bound $F_{1,N}$}

Equation $[N]$ of \eqref{IfantisId2bis} yields $F_{1,N} = |c_N^-|$. This gives
\eqref{IfantisThm2} because, due to the divergence of $b_N$ and the boundedness of $\lambda_N$, we can ensure that $\lim_{N\to\infty}|\lambda_N-b_N|/b_N=1$ so that $|\lambda_N-b_N|$ can be substituted by $b_N$ when imposing $\lim_{N\to\infty}a_NF_{1,N}=0$.

\smallskip

\noindent \underline{The bound $F_{2,N}$}

Inserting equation $[N-1]$ into equation $[N]$ leads to
 \begin{equation}
 \delta_N = c_N^-c_{N-1}^-\delta_{N-2} + c_N^-c_{N-1}^+\delta_N.
 \label{m=2}
 \end{equation}
Thus we can take $F_{2,N} = |c_N^-c_{N-1}^-|+|c_N^-c_{N-1}^+|$, and $\lim_{N\to\infty}a_NF_{2,N}=0$ is equivalent to  $\lim_{N\to\infty}a_N|c_N^-c_{N-1}^-|=\lim_{N\to\infty}a_N|c_N^-c_{N-1}^+|=0$. This ends in \eqref{IfantisThm3} when substituting in both asymptotic conditions the factors $|\lambda_N-b_N|$ and $|\lambda_N-b_{N-1}|$ by the equivalent ones $b_N$ and $b_{N-1}$.

\smallskip

\noindent \underline{The bound $F_{3,N}$}

Now we introduce equations $[N-2]$ and $[N]$ into \eqref{m=2} obtaining
\begin{equation}
 \delta_N = c_N^-c_{N-1}^-c_{N-2}^-\delta_{N-3}
 + c_N^-c_{N-1}^-c_{N-2}^+\delta_{N-1}
 + c_N^-c_{N-1}^+c_N^-\delta_{N-1}.
 \label{m=3}
 \end{equation}
This gives $F_{3,N}=|c_N^-c_{N-1}^-c_{N-2}^-|
 + |c_N^-c_{N-1}^-c_{N-2}^+|
 + |c_N^-c_{N-1}^+c_N^-|$, from which \eqref{IfantisThm4} is obtained analogously to the previous cases.

\smallskip

Reference \cite{IKP2007} stops the procedure at this stage, but it is clear that it can continue indefinitely providing infinitely many conditions for  $\Lambda(T)\subseteq\sigma(T)$. For instance, the next bound and sufficient conditions are shown below.

\smallskip

\noindent \underline{The bound $F_{4,N}$}

Inserting equations $[N-3]$ and $[N-1]$ into \eqref{m=3} yields
 \begin{equation}
 \begin{aligned}
 \kern-7pt \text{\footnotesize $\delta_N$} \,
 & \text{\footnotesize
 $= c_N^-c_{N-1}^-c_{N-2}^-c_{N-3}^-\delta_{N-4}
 + c_N^-c_{N-1}^-c_{N-2}^-c_{N-3}^+\delta_{N-2}
 + c_N^-c_{N-1}^-c_{N-2}^+c_{N-1}^-\delta_{N-2}$}
 \\
 & \kern10pt \text{\footnotesize
 $+ \; c_N^-c_{N-1}^-c_{N-2}^+c_{N-1}^+\delta_N
 + c_N^-c_{N-1}^+c_N^-c_{N-1}^-\delta_{N-2}
 + c_N^-c_{N-1}^+c_N^-c_{N-1}^+\delta_N.$}
 \label{m=4}
 \end{aligned}
 \end{equation}
The bound $F_{4,N}=$ {\footnotesize$|c_N^-c_{N-1}^-c_{N-2}^-c_{N-3}^-|
 + |c_N^-c_{N-1}^-c_{N-2}^-c_{N-3}^+|
 + |c_N^-c_{N-1}^-c_{N-2}^+c_{N-1}^-|
 + |c_N^-c_{N-1}^-c_{N-2}^+c_{N-1}^+|
 + |c_N^-c_{N-1}^+c_N^-c_{N-1}^-|
 + |c_N^-c_{N-1}^+c_N^-c_{N-1}^+|$} leads to a new condition which, together with $\lim_{n\to\infty}b_n=\infty$, guarantees $\Lambda(T)\subseteq\sigma(T)$, namely,
 \begin{equation}
 \lim_{n\to\infty}
 \text{\footnotesize $\frac{a_n a_{n-1}}{b_nb_{n-1}}
 \left[
 \left(\frac{a_{n-1}^{2}}{b_n}+\frac{a_{n-2}^{2}}{b_{n-2}}\right)
 \frac{a_{n-1}+a_{n-2}}{b_{n-1}}
 +\frac{a_{n-2}a_{n-3}(a_{n-3}+a_{n-4})}{b_{n-2}b_{n-3}}\right]$}
 =0.
 \label{new}
 \end{equation}

The inclusion $\Lambda(T)\subseteq\sigma(T)$ remains true assuming that $\lim_{n\to\infty}|b_n|=\infty$ and substituting $b_n$ by $|b_n|$ in conditions \eqref{IfantisThm2}--\eqref{IfantisThm4} and \eqref{new}. The reason for this is that $\lim_{N\to\infty}|\lambda_N-b_{N-k}|/|b_{N-k}|=1$ under the divergence of $|b_n|$, thus $|\lambda_N-b_{N-k}|$ can be substituted by $|b_{N-k}|$ in \eqref{IfantisId5}.
Moreover, Theorem~\ref{thm-self} and Remark~\ref{rem-interlace} prove that these conditions guarantee the self-adjointness of $T$, so that actually they imply that $\Lambda(T)=\sigma(T)$.
Therefore, we have found the following improvement of the results in \cite{IKP2007}.

\begin{thm} \label{Ifantis-update}
Let $(\mathfrak{B}_1)$--$(\mathfrak{B}_4)$ be the conditions obtained respectively from \eqref{IfantisThm2}--\eqref{IfantisThm4} and \eqref{new} when substituting $b_n$ by $|b_n|$.
If $\lim_{n\to\infty}|b_n|=\infty$, any of the conditions $(\mathfrak{B}_1)$--$(\mathfrak{B}_4)$ implies that $T$ is self-adjoint and $\Lambda(T)=\sigma(T)$.
\end{thm}

As it is pointed out in \cite{IKP2007}, none of the conditions \eqref{IfantisThm2} or \eqref{IfantisThm3} is weaker than the other. Indeed, we will see that this also holds for $(\mathfrak{B}_m)$, $1 \le m \le 4$, and in general for the conditions obtained from any bound $F_{m,N}$, which become complementary (see \S\ref{Comparisons}). Therefore, the sufficient conditions obtained from all the bounds $F_{m,N}$ are in principle of equal interest.

It can be argued that the results for large values of $m$ are of doubtful utility because the complexity of the sufficient conditions grows quickly as $m$ gets bigger. Nevertheless, as we will see in \S\ref{Comparisons}, this does not prevent from applying succesfully these infinitely many conditions to concrete examples. 

Indeed, \S\ref{GmC} shows that it is possible to extract simple but quite general consequences of interest (see Theorem \ref{weak}) from the whole set of complicated statements that appear for all the values of $m$. To understand the idea in a simple setting, we will first explain it using Theorem \ref{Ifantis-update}. The expressions involved in ($\mathfrak{B}_1$)--($\mathfrak{B}_3$) can be split as
 \[
 \begin{array}{l}
 \frac{a_na_{n-1}}{|b_n|}
 = \frac{a_n}{|b_n|^{1/2}} \frac{a_{n-1}}{|b_n|^{1/2}},
 \smallskip \\
 \frac{a_n a_{n-1}(a_{n-1}+a_{n-2})}{|b_n| |b_{n-1}|}
 = \frac{a_n}{|b_n|^{2/3}}
 \left(\frac{a_{n-1}}{|b_n|^{2/3}}\right)^{\scriptscriptstyle\kern-2pt 1/2}
 \left[
 \left(\frac{a_{n-1}}{|b_{n-1}|^{2/3}}\right)^{\scriptscriptstyle\kern-2pt 3/2}
 + \left(\frac{a_{n-1}}{|b_{n-1}|^{2/3}}\right)^{\scriptscriptstyle\kern-2pt 1/2}
 \frac{a_{n-2}}{|b_{n-1}|^{2/3}}
 \right],
 \smallskip \\
 \frac{a_n a_{n-1}}{|b_n||b_{n-1}|}
 \left[
 \frac{a_{n-1}^{2}}{|b_n|}+\frac{a_{n-2}(a_{n-2}+a_{n-3})}{|b_{n-2}|}
 \right]
 = \frac{a_n}{|b_n|^{3/4}}
 \left(\frac{a_{n-1}}{|b_n|^{3/4}}\right)^{\scriptscriptstyle\kern-2pt 1/3}
 \left\{
 \left(\frac{a_{n-1}}{|b_{n-1}|^{3/4}}\right)^{\scriptscriptstyle\kern-2pt 4/3}
 \left(\frac{a_{n-1}}{|b_n|^{3/4}}\right)^{\scriptscriptstyle\kern-2pt 4/3}
 \right.
 \smallskip \\ \kern162pt
 + \left(\frac{a_{n-1}}{|b_{n-1}|^{3/4}}\right)^{\scriptscriptstyle\kern-2pt 2/3}
 \left[
 \left(\frac{a_{n-2}}{|b_{n-1}|^{3/4}}\right)^{\scriptscriptstyle\kern-2pt 2/3}
 \left(\frac{a_{n-2}}{|b_{n-2}|^{3/4}}\right)^{\scriptscriptstyle\kern-2pt 4/3}
 \right.
 \smallskip \\ \kern160pt
 \left.\left.
 + \left(\frac{a_{n-2}}{|b_{n-1}|^{3/4}}\right)^{\scriptscriptstyle\kern-2pt 2/3}
 \left(\frac{a_{n-2}}{|b_{n-2}|^{3/4}}\right)^{\scriptscriptstyle\kern-2pt 1/3}
 \frac{a_{n-3}}{|b_{n-2}|^{3/4}}
 \right]
 \right\},
 \end{array}
 \]
while the analogous expression in ($\mathfrak{B}_4$) can be written as a sum of terms which are products of positive powers of
 \[
 \frac{a_{n-k}}{|b_{n-k}|^{4/5}},
 \quad
 \frac{a_{n-k-1}}{|b_{n-k}|^{4/5}},
 \quad k=0,1,2,3.
 \]
This splitting shows that
 \[
 a_n, a_{n-1} = o(|b_n|^\frac{m}{m+1})
 \; \Rightarrow \;
 \eqref{eq-strong}, \qquad m=1,2,3,4,
 \]
where the usual notation $y_n=o(z_n)$ stands for $\lim_{n\to\infty}y_n/z_n=0$.

Therefore, a weaker but much simpler version of Theorem \ref{Ifantis-update} states that $T$ is self-adjoint and $\Lambda(T)=\sigma(T)$ provided that $|b_n|$ diverges and $a_n, a_{n-1} = o(|b_n|^\frac{m}{m+1})$ for some of the values $m=1,2,3,4$. Actually, in contrast to conditions \eqref{eq-strong}, the simpler ones $a_n, a_{n-1} = o(|b_n|^\frac{m}{m+1})$ are not complementary, but they become weaker as $m$ gets bigger. This means that the weak version of Theorem \ref{Ifantis-update} can be summarized by the single result for the biggest value $m=4$,
 \[
 \left\{ \kern-3pt
 \begin{array}{l}
 {\displaystyle\lim_{n\to\infty}}|b_n|=\infty
 \\
 a_n, a_{n-1} = o(|b_n|^{4/5})
 \end{array}
 \right.
 \Rightarrow\; T \text{ is self-adjoint and } \Lambda(T)=\sigma(T).
 \]

This suggests that, in the weak version, it should be possible to enclose the information given by the results for all the values of $m$ into a single statement. Such a statement is among the objectives of the next section, which is devoted to the extension of the previous results to any index $m$.

 \section{General $m$-conditions for $\Lambda(T)=\sigma(T)$ and self-adjointness}
 \label{GmC}
 \setcounter{equation}{0}
 \renewcommand{\theequation}{\thesection.\arabic{equation}}
 \setcounter{thm}{0}
 \renewcommand{\thethm}{\thesection.\arabic{thm}}
 \par
To deal with the bounds $F_{m,N}$ for any value of $m$ we first need an expression for $\delta_N$ generalizing \eqref{m=2}, \eqref{m=3} and \eqref{m=4}, i.e. an expression obtained using recursively the last $m$ equations of \eqref{IfantisId2}. For this purpose we introduce the multi-indices
$\boldsymbol{j}_m=(j_1,j_2,\dots,j_m)$, $j_s\in\mathbb{Z}$, and the sets
 \begin{equation}
 \begin{aligned}
 & \mathcal{I}_m
 = \{(\boldsymbol{j}_m|\boldsymbol{k}_m) \, :
 \, j_1=0, \;
 k_s=j_s=j_{s+1}+1 \, \text{ or } \, k_s=j_s+1=j_{s+1}\},
 \\
 & \widehat{\mathcal{I}}_m
 = \{(\boldsymbol{j}_{m+1}|\boldsymbol{k}_m) \, :
 \, j_1=0, \;
 k_s=j_s=j_{s+1}+1 \, \text{ or } \, k_s=j_s+1=j_{s+1}\},
 \\
 & \mathcal{I}_m^+
 = \{(\boldsymbol{j}_m|\boldsymbol{k}_m) \in \mathcal{I}_m \, :
 \, j_s\ge0, \; k_s\ge1\},
 \\
 & \widehat{\mathcal{I}}_m^+
 = \{(\boldsymbol{j}_{m+1}|\boldsymbol{k}_m) \in \widehat{\mathcal{I}}_m \, :
 \, j_s\ge0, \; k_s\ge1\}.
 \end{aligned}
 \label{indices}
 \end{equation}
Using this notation we have the following result.

\begin{prop} \label{deltaN}
For any $m\in\{1,2,\dots,N\}$, the solutions $\delta_k$ of \eqref{IfantisId2} satisfy
 \begin{equation}
 \delta_N = \kern-10pt \sum_{
 (\boldsymbol{j}_{m+1}|\boldsymbol{k}_m)\in\widehat{\mathcal{I}}_m^+}
 \frac{a_{N-k_1}}{\lambda_N-b_{N-j_1}} \frac{a_{N-k_2}}{\lambda_N-b_{N-j_2}}
 \cdots \frac{a_{N-k_m}}{\lambda_N-b_{N-j_m}} \, \delta_{N-j_{m+1}},
 \label{eq-deltaN}
 \end{equation}
provided that $\lambda_N\neq b_{N-j}$ for $j=0,1,\dots,m-1$.
\end{prop}

\begin{proof}
Let us proceed by induction on $m$. Equation $[N]$ of \eqref{IfantisId2bis} is directly the result for $m=1$ because $\widehat{\mathcal{I}}_1^+=\{(0,1|1)\}$.

Assume now \eqref{eq-deltaN} for an index $m<N$. Then, $0 \le j_{m+1} \le m$ for each element of the set $\widehat{\mathcal{I}}_m^+$, so that $\lambda_N \ne b_{N-j_{m+1}}$ under the hypothesis of the theorem. Since $N \ge N-j_{m+1} \ge N-m > 1$, it makes sense to use equation $[N-j_{m+1}]$ of \eqref{IfantisId2bis}. Inserting it into each summand of \eqref{eq-deltaN} and using the convention $\delta_{N+1}=0$ gives
 \[
 \begin{aligned}
 \delta_N & = \kern-15pt \sum_{
 \substack{(\boldsymbol{j}_{m+1}|\boldsymbol{k}_m)\in\widehat{\mathcal{I}}_m^+
 \\
 k_{m+1}=j_{m+1}=j_{m+2}+1
 \smallskip \\
 \text{or}
 \\
 k_{m+1}=j_{m+1}+1=j_{m+2}
 }}
 \frac{a_{N-k_1}}{\lambda_N-b_{N-j_1}}
 \cdots \frac{a_{N-k_m}}{\lambda_N-b_{N-j_m}}
 \frac{a_{N-k_{m+1}}}{\lambda_N-b_{N-j_{m+1}}}
 \, \delta_{N-j_{m+2}}
 \\
 & = \kern-10pt \sum_{
 (\boldsymbol{j}_{m+2}|\boldsymbol{k}_{m+1})\in\widehat{\mathcal{I}}_{m+1}^+
 }
 \frac{a_{N-k_1}}{\lambda_N-b_{N-j_1}}
 \cdots \frac{a_{N-k_m}}{\lambda_N-b_{N-j_m}}
 \frac{a_{N-k_{m+1}}}{\lambda_N-b_{N-j_{m+1}}}
 \, \delta_{N-j_{m+2}},
 \end{aligned}
 \]
which proves the result for the index $m+1$.
\end{proof}

As a direct consequence of the previous proposition, we find a general expression for the bound $F_{m,N}$.

\begin{prop} \label{bound}
Given $m\in\mathbb{N}$, for any $N \ge m$, the $N$-th coordinate $\delta_N$ of the normalized eigenvector of $T_N$ with eigenvalue $\lambda_N$ is bounded by
 \begin{equation}
 F_{m,N} = \kern-7pt
 \sum_{(\boldsymbol{j}_m|\boldsymbol{k}_m)\in\mathcal{I}_m^+}
 \frac{a_{N-k_1}}{|\lambda_N-b_{N-j_1}|} \frac{a_{N-k_2}}{|\lambda_N-b_{N-j_2}|}
 \cdots \frac{a_{N-k_m}}{|\lambda_N-b_{N-j_m}|},
 \label{eq-bound}
 \end{equation}
provided that $\lambda_N\neq b_{N-j}$ for $j=0,1,\dots,m-1$.
\end{prop}

The above expression of the bound $F_{m,N}$ leads to the generalization of Theorem \ref{Ifantis-update} for any value of $m$.

\begin{thm} \label{strong}
For any $m\in\mathbb{N}$, the conditions
 \begin{align}
 & \lim_{n\to\infty}|b_n|=\infty,
 \notag
 \\
 & \lim_{n\to\infty} a_n G_{m,n}^+=0,
 \quad
 G_{m,n}^+ = \kern-7pt
 \sum_{(\boldsymbol{j}_m|\boldsymbol{k}_m)\in\mathcal{I}_m^+}
 \frac{a_{n-k_1}}{|b_{n-j_1}|} \frac{a_{n-k_2}}{|b_{n-j_2}|}
 \cdots \frac{a_{n-k_m}}{|b_{n-j_m}|},
 \label{eq-strong} \tag{$\mathfrak{B}_m$}
 \end{align}
imply that $T$ is self-adjoint and $\Lambda(T)=\sigma(T)$.
\end{thm}

\begin{proof}
In view of Theorem \ref{thm-self} and Remark \ref{rem-interlace}, it is enough to prove that the hypothesis of the theorem yield $\lim_{N\to\infty} a_NF_{m,N}=0$ for any bounded sequence $\lambda_N$, where we can assume the expression \eqref{eq-bound} of $F_{m,N}$ because it is valid for $N$ big enough due to the divergence of $|b_{N-j}|$ as $N\to\infty$. Due to the positivity of the summands of $F_{m,N}$ and $G_{m,N}^+$, we have the equivalences
 \[
 \kern-4pt
 \begin{aligned}
 & \lim_{N\to\infty} \kern-2pt a_N F_{m,N} = 0
 \,\Leftrightarrow\kern-1pt
 \lim_{N\to\infty} \kern-2pt
 \text{\footnotesize $a_N \frac{a_{N-k_1}}{|\lambda_N-b_{N-j_1}|}
 \cdots \frac{a_{N-k_m}}{|\lambda_N-b_{N-j_m}|}$} = 0,
 \kern3pt \forall (\boldsymbol{j}_m|\boldsymbol{k}_m)\in\mathcal{I}_m^+
 \\
 & \Leftrightarrow
 \lim_{N\to\infty}
 \text{\footnotesize $a_N \frac{a_{N-k_1}}{|b_{N-j_1}|}
 \cdots \frac{a_{N-k_m}}{|b_{N-j_m}|}$} = 0,
 \kern5pt \forall (\boldsymbol{j}_m|\boldsymbol{k}_m)\in\mathcal{I}_m^+
 \;\Leftrightarrow
 \lim_{N\to\infty} a_NG_{m,N}^+ = 0,
 \end{aligned}
 \]
which prove the theorem.
\end{proof}

The $m$-conditions \eqref{eq-strong} are the generalization of the conditions in Theorem \ref{Ifantis-update} which appear for the particular sets of multi-indices
 \[
 \begin{aligned}
 & \scriptstyle
 \mathcal{I}_1^+=\{(0|1)\},
 \quad
 \mathcal{I}_2^+=\{(0,1|1,1),(0,1|1,2)\},
 \quad
 \mathcal{I}_3^+=\{(0,1,0|1,1,1),(0,1,2|1,2,2),(0,1,2|1,2,3)\},
 \\
 & \scriptstyle
 \mathcal{I}_4^+=\{(0,1,0,1|1,1,1,1),(0,1,0,1|1,1,1,2),(0,1,2,1|1,2,2,1),
 (0,1,2,1|1,2,2,2),(0,1,2,3|1,2,3,3),(0,1,2,3|1,2,3,4)\}.
 \end{aligned}
 \]

Although Theorem \ref{strong} can be particularized to any other value of $m$, it is also possible to extract information of interest from the general $m$-conditions without resorting to intricate asymptotic conditions.

\begin{thm} \label{weak}
The conditions
 \[
 \begin{aligned}
 & \lim_{n\to\infty}|b_n|=\infty,
 \\
 & a_n,a_{n-1}=o(|b_n|^r) \text{ for some } r<1,
 \end{aligned}
 \]
imply that $T$ is self-adjoint and $\Lambda(T)=\sigma(T)$.
\end{thm}

\begin{proof}
Let $m\in\mathbb{N}$. For any $(\boldsymbol{j}_m,\boldsymbol{k}_m)\in\mathcal{I}_m^+$, consider the factorization
 \[
 \begin{aligned}
 a_n \, \frac{a_{n-k_1}}{|b_{n-j_1}|} \cdots \frac{a_{n-k_m}}{|b_{n-j_m}|}
 & = \frac{a_n}{|b_{n-j_1}|^{\frac{m}{m+1}}}
 \frac{a_{n-k_m}}{|b_{n-j_m}|^{\frac{m}{m+1}}}
 \prod_{s=1}^{m-1}
 \frac{a_{n-k_s}}{|b_{n-j_s}|^{\frac{s}{m+1}} |b_{n-j_{s+1}}|^{\frac{m-s}{m+1}}}
 \\
 & \kern-37pt = \frac{a_n}{|b_{n}|^{\frac{m}{m+1}}}
 \frac{a_{n-k_m}}{|b_{n-j_m}|^{\frac{m}{m+1}}}
 \prod_{s=1}^{m-1}
 \left(
 \frac{a_{n-k_s}}{|b_{n-j_s}|^{\frac{m}{m+1}}}
 \right)^{\kern-3pt\frac{s}{m}}
 \left(
 \frac{a_{n-k_s}}{|b_{n-j_{s+1}}|^{\frac{m}{m+1}}}
 \right)^{\kern-3pt\frac{m-s}{m}},
 \end{aligned}
 \]
where we have taken into account that $j_1=0$ in $\mathcal{I}_m^+$.
From the definition of the set $\mathcal{I}_m^+$ we see that $n-k_s=n-j_s=n-j_{s+1}-1$ or $n-k_s=n-j_s-1=n-j_{s+1}$ in this factorization. Thus, the condition $a_n,a_{n-1}=o(|b_n|^\frac{m}{m+1})$ guarantees that all the summands of $G_{m,n}^+$ in \eqref{eq-strong} converge to zero as $n\to\infty$. Bearing in mind Theorem \ref{strong}, this means that
 \[
 \left\{
 \begin{array}{l}
 {\displaystyle\lim_{n\to\infty}}|b_n|=\infty
 \\
 a_n,a_{n-1}=o(|b_n|^\frac{m}{m+1}) \text{ for some } m\in\mathbb{N}
 \end{array}
 \right.
 \Rightarrow\kern5pt
 \text{\parbox{100pt}{$T$ is self-adjoint and $\Lambda(T)=\sigma(T)$.}}
 \]
This statement is equivalent to the theorem because $m/(m+1)$ is an increasing sequence converging to 1 for $m\to\infty$. Thus, for any $r<1$ there exists $m\in\mathbb{N}$ such that $r<m/(m+1)$, and then the divergence of $|b_n|$ ensures that the asymptotic behaviour $o(|b_n|^r)$ implies $o(|b_n|^\frac{m}{m+1})$.
\end{proof}

\section{Other $m$-conditions for self-adjointness}
 \label{OmC}
 \setcounter{equation}{0}
 \renewcommand{\theequation}{\thesection.\arabic{equation}}
 \setcounter{thm}{0}
 \renewcommand{\thethm}{\thesection.\arabic{thm}}
 \par

We have seen that the study of the relation between $\Lambda(T)$ and $\sigma(T)$ sheds light on the self-adjointness of $T$. We will show in this section that the iterative use of eigenvalue equations to obtain sets of infinitely many sufficient conditions for self-adjointness (and thus for $\Lambda(T) \supseteq \sigma(T)$) can be pursued in other ways. Two different types of results will illustrate this strategy. Although none of them deals with the limit points $\Lambda(T)$, both have in common with the previous approach the fact that they are especially adapted to the analysis of symmetric tridiagonal operators with an unbounded main diagonal.

We will discuss first a set of $m$-conditions extending the well known Carleman criterion, which states that $T$ is self-adjoint if
 \begin{equation} \label{CAR}
 \sum_{n=1}^\infty \frac{1}{a_n} = \infty.
 \end{equation}
To obtain this generalization, let us remember first a proof of Carleman's criterion based on an orthogonal polynomial characterization of self-adjointness: $T$ is self-adjoint iff $\sum_{n=1}^\infty |p_n(z)|^2$ is divergent, where $p_n(x)$ are the orthogonal polynomials given in \eqref{Intro_OPRR} and $z$ is any point of $\mathbb{C}\setminus\mathbb{R}$ (we will eventually choose $z=i$ for convenience). This is equivalent to saying that $T$ is not self-adjoint iff $(p_1(z),p_2(z),\dots)$ is in $\ell^2$, which means that $(p_1(z),p_2(z),\dots)$ is an eigenvector of the maximal extension of $T$ with eigenvalue $z$. Recurrence \eqref{Intro_OPRR} is in this case the corresponding eigenvalue equation.

The Christoffel-Darboux identity for orthonormal polynomials,
 \[
 (x-y) \sum_{k=1}^n p_k(x) p_k(y) = a_n (p_{n+1}(x)p_n(y)-p_n(x)p_{n+1}(y)),
 \]
yields for $x=i$ and $y=-i$
 \[
 \sum_{k=1}^n |v_k|^2 = a_n \operatorname{Im}(v_{n+1}\overline{v_n}),
 \qquad v_n=p_n(i).
 \]
From this identity we obtain
 \begin{equation}
 1 \le a_n |v_n| |v_{n+1}|,
 \label{CD0}
 \end{equation}
so, due to the Cauchy-Schwarz inequality,
 \[
 \sum_{n=1}^\infty \frac{1}{a_n} \le \sum_{n=1}^\infty |v_n||v_{n+1}|
 \le \sum_{n=1}^\infty |v_n|^2.
 \]
Therefore, if $\sum_{n=1}^\infty 1/a_n$ diverges, so does $\sum_{n=1}^\infty |v_n|^2$ and $T$ is self-adjoint.

\smallskip

Relations \eqref{Intro_OPRR} defining $p_n(x)$ will play the role of `eigenvalue' equations to generate new versions of Carleman's criterion. These new criteria coming from the iterative use of \eqref{Intro_OPRR} amount to the substitution of the condition $\sum_{n=1}^\infty 1/a_n=\infty$ by $m$-conditions with the form
 \[
 \sum_{n=m+1}^\infty \frac{1}{a_nG_{m,n}} = \infty,
 \qquad
 G_{m,n} = G_{m,n}(\{a_{n-k}\}_{k=-m+1}^m,\{b_{n-k}\}_{k=-m+1}^{m-1}).
 \]
As we will see, $G_{m,n}=\infty$ when $b_{n-k}=0$ for some $k$ with $|k|\le m-1$. Then we understand that $1/a_nG_{m,n}=0$ in the above series.

To generate these Carleman type criteria, note that \eqref{Intro_OPRR} gives the inequality
 \begin{equation}
 |v_n| \le \gamma_n^- |v_{n-1}| + \gamma_n^+ |v_{n+1}|,
 \quad \gamma_n^- = \frac{a_{n-1}}{|b_n|}, \quad \gamma_n^+=\frac{a_n}{|b_n|},
 \quad n\ge1,
 \label{ineq_vn}
 \end{equation}
where $v_0=0$ due to the convention $p_0(x)=0$, and we take $\gamma_n^\pm=\infty$ if $b_n=0$. From \eqref{ineq_vn} we obtain
 \[
 |v_n| \le (\gamma_n^-+\gamma_n^+) (|v_{n-1}|+|v_{n+1}|),
 \qquad n\ge2,
 \]
which combined with \eqref{CD0} leads to
 \[
 \frac{1}{a_n(\gamma_n^-+\gamma_n^+)} \le (|v_{n-1}|+|v_{n+1}|)|v_{n+1}|,
 \qquad n\ge2,
 \]
assuming that $1/a_n(\gamma_n^-+\gamma_n^+)=0$ if $b_n=0$.
Using again the Cauchy-Schwarz inequality we find that
 \[
 \sum_{n=2}^\infty \frac{1}{a_n(\gamma_n^-+\gamma_n^+)}
 \le \sum_{n=2}^\infty (|v_{n-1}||v_{n+1}|+|v_{n+1}|^2)
 \le 2 \sum_{n=1}^\infty |v_n|^2.
 \]
Just as in the case of Carleman's criterion we arrive at the following result.

\begin{thm} \label{Car1}
$T$ is self-adjoint if
 \begin{equation}
 \sum_{n=2}^\infty \frac{1}{a_n\left(\frac{a_{n-1}+a_n}{|b_n|}\right)} = \infty.
 \label{eq-Car1} \tag{$\mathfrak{C}_1$}
 \end{equation}
\end{thm}

Condition \eqref{eq-Car1} is similar to a known one introduced by J. J. Dennis and H. S. Wall in \cite[Theorem 4.4]{DW1945} regarding the study of Jacobi continued fractions \eqref{Intro_CF} with complex coefficients $a_n$, $b_n$. In the case $a_n>0$, $b_n\in\mathbb{R}$, Dennis-Wall condition also implies the self-adjointness of $T$ and reads as
 \begin{equation}
 \sum_{n=2}^\infty \frac{|b_n|}{a_na_{n-1}} = \infty.
 \label{eq-DW}
 \end{equation}
This condition is obviously weaker than \eqref{eq-Car1}, but this latter one should be seen as a first instance of a set of infinitely many self-adjointness conditions which can eventually improve the results obtained solely with Carleman and Dennis-Wall criteria. Although these two criteria are also the simplest of infinitely many ones given in \cite[Equation (4.14)]{DW1945}, the advantage of the conditions that we will obtain when generalizing \eqref{eq-Car1} rests on their controllable dependence on the coefficients $a_n$, $b_n$. This makes possible a simultaneous application of the infinitely many Carleman type criteria in practical cases, as it is illustrated in
\S\ref{Comparisons}.

We can obtain another variant of Carleman's criterion by inserting the $n-1$-th and $n+1$-th equations of \eqref{ineq_vn} into the $n$-th one,
 \[
 |v_n| \le
 \gamma_n^- (\gamma_{n-1}^-|v_{n-2}|+\gamma_{n-1}^+|v_n|)
 + \gamma_n^+ (\gamma_{n+1}^-|v_n|+\gamma_{n+1}^+|v_{n+2}|),
 \qquad n\ge3,
 \]
which leads to
 \[
 |v_n| \le
 [\gamma_n^- (\gamma_{n-1}^-+\gamma_{n-1}^+)
 + \gamma_n^+ (\gamma_{n+1}^-+\gamma_{n+1}^+)]
 (|v_{n-2}|+|v_n|+|v_{n+2}|),
 \qquad n\ge3.
 \]
The above inequality can be combined with \eqref{CD0} to obtain for $n\ge3$
 \[
 \frac{1}
 {a_n[\gamma_n^- (\gamma_{n-1}^-+\gamma_{n-1}^+)
 + \gamma_n^+ (\gamma_{n+1}^-+\gamma_{n+1}^+)]}
 \le (|v_{n-2}|+|v_n|+|v_{n+2}|) |v_{n+1}|,
 \]
which becomes trivial when $b_{n-1}$, $b_n$ or $b_{n+1}$ vanish because then the left-hand side must be understood as zero. Proceeding as in the previous case we get
 \[
 \sum_{n=3}^\infty \frac{1}{a_n
 [\gamma_n^-(\gamma_{n-1}^-+\gamma_{n-1}^+)
 + \gamma_n^+(\gamma_{n+1}^-+\gamma_{n+1}^+)]}
 \le 3 \sum_{n=1}^\infty |v_n|^2,
 \]
which ends in a new Carleman type criterion.

\begin{thm} \label{Car2}
$T$ is self-adjoint if
 \begin{equation}
 \sum_{n=3}^\infty \frac{1}{a_n
 \left[\frac{a_{n-1}}{|b_n|} \left(\frac{a_{n-2}+a_{n-1}}{|b_{n-1}|}\right)
 + \frac{a_n}{|b_n|} \left(\frac{a_n+a_{n+1}}{|b_{n+1}|}\right)\right]}
 = \infty.
 \label{eq-Car2}
 \tag{$\mathfrak{C}_2$}
 \end{equation}
\end{thm}

Conditions \eqref{eq-Car1} and \eqref{eq-Car2} are only two particular cases of a set of $m$-conditions generalizing Carleman's criterion. In contrast to Theorem \ref{strong}, these $m$-conditions involve the full set of multi-indices $\mathcal{I}_m$ instead of its subset $\mathcal{I}_m^+$ because we are dealing now with `eigenvalue' equations related with the full operator $T$ instead of a truncated operator.

\begin{thm} \label{Carm}
For any $m\in\mathbb{N}$, the condition
 \begin{equation}
 \sum_{n=m+1}^\infty \frac{1}{a_nG_{m,n}} = \infty,
 \quad
 G_{m,n} = \kern-7pt
 \sum_{(\boldsymbol{j}_m|\boldsymbol{k}_m)\in\mathcal{I}_m}
 \frac{a_{n-k_1}}{|b_{n-j_1}|} \frac{a_{n-k_2}}{|b_{n-j_2}|}
 \cdots \frac{a_{n-k_m}}{|b_{n-j_m}|},
 \label{eq-Carm}
 \tag{$\mathfrak{C}_m$}
 \end{equation}
implies that $T$ is self-adjoint.
\end{thm}

\begin{proof}
From \eqref{ineq_vn}, a proof by induction similar to that one of Proposition \ref{deltaN} shows that
 \begin{equation}
 |v_n| \le
 \sum_{(\boldsymbol{j}_{m+1}|\boldsymbol{k}_m)\in\widehat{\mathcal{I}}_m}
 \frac{a_{n-k_1}}{|b_{n-j_1}|} \frac{a_{n-k_2}}{|b_{n-j_2}|}
 \cdots \frac{a_{n-k_m}}{|b_{n-j_m}|} \, |v_{n-j_{m+1}}|,
 \quad n\ge m+1.
 \label{vn}
 \end{equation}
Bearing in mind that $j_{m+1}\in\{-m,-m+2,\dots,m-2,m\}$ for $(\boldsymbol{j}_{m+1}|\boldsymbol{k}_m)\in\widehat{\mathcal{I}}_m$, inequality \eqref{vn} yields
 \[
 |v_n| \le G_{m,n} \sum_{k=0}^m |v_{n-m+2k}|,
 \qquad n\ge m+1.
 \]
This can be combined with \eqref{CD0} to obtain
 \[
 \frac{1}{a_n G_{m,n}} \le
 \sum_{k=0}^m |v_{n-m+2k}| |v_{n+1}|,
 \qquad n\ge m+1,
 \]
a trivial inequality when $G_{m,n}=\infty$ because we understand that $1/a_nG_{m,n}=0$ in such a case. Then, the Cauchy-Schwarz inequality gives
 \[
 \sum_{n=m+1}^\infty \frac{1}{a_n G_{m,n}} \le (m+1) \sum_{n=1}^\infty |v_n|^2,
 \]
which proves the theorem.
\end{proof}

We will see in \S\ref{Comparisons} that Theorem~\ref{Carm} applies in some cases where conditions \eqref{CAR} and \eqref{eq-DW} do not. This shows that Carleman type $m$-conditions \eqref{eq-Carm} can be used to improve the results obtained with the standard Carleman and Dennis-Wall criteria.

Since the general term of a convergent series must converge to zero, a consequence follows directly from Theorem~\ref{Carm}.

\begin{cor} \label{cor-Car}
For any $m\in\mathbb{N}$, the condition
 \[
 \liminf_{n\to\infty} a_nG_{m,n} < \infty
 \]
implies that $T$ is self-adjoint.
\end{cor}

The orthogonal polynomial characterization of self-adjointness can be used to generate another type of $m$-conditions for self-adjointness. The starting point is again a consequence of \eqref{Intro_OPRR}, namely,
 \begin{equation}
 |v_n|^2 \le
 2 \left[(\gamma_n^-)^2 |v_{n-1}|^2 + (\gamma_n^+)^2 |v_{n+1}|^2\right],
 \label{ineq_vn_bis}
 \end{equation}
which holds for any $n\in\mathbb{Z}$ if we define $\gamma_n^-=0$ for $n\le1$ and $\gamma_n^+=v_n=0$ for $n\le0$. For the rest of the indices we should take $\gamma_n^\pm=\infty$ when $b_n=0$.

Summing up \eqref{ineq_vn_bis} for $n\ge1$ gives
 \[
 \sum_{n=1}^\infty |v_n|^2 \le
 2 \sum_{n=1}^\infty
 \left[(\gamma_{n+1}^-)^2 + (\gamma_{n-1}^+)^2\right] |v_n|^2.
 \]
Therefore, the inequality
 \begin{equation}
 (\gamma_{n+1}^-)^2 + (\gamma_{n-1}^+)^2 < \frac{1}{2}, \qquad n\ge1,
 \label{strong-JN1}
 \end{equation}
is incompatible with the convergence of $\sum_{n=1}^\infty|v_n|^2$ and implies that $T$ is self-adjoint.

Suppose now that \eqref{ineq_vn_bis} holds only for $n$ big enough.
Then, $b_n\ne0$ up to a finite number of indices $n$. We can define a new symmetric tridiagonal operator $\widetilde{T}$ satisfying \eqref{strong-JN1} by changing the null coefficients $b_k$ of $T$ by non null ones, and setting $a_k$ small enough for all the coefficients appearing in the expressions $(\gamma_{n+1}^-)^2+(\gamma_{n-1}^+)^2$ where \eqref{strong-JN1} fails, the rest of the coefficients coinciding with those of $T$. The operator $\widetilde{T}$ is self-adjoint because it satisfies \eqref{strong-JN1}. Since $T$ differs from $\widetilde{T}$ in a bounded self-adjoint operator, we conclude that $T$ is self-adjoint too.

Thus, we have proved the following result.

\begin{thm} \label{thm-JN1}
$T$ is self-adjoint if there exists an index $n_0\in\mathbb{N}$ such that
 \begin{equation}
 \frac{a_n^2}{b_{n+1}^2}+\frac{a_{n-1}^2}{b_{n-1}^2} < \frac{1}{2},
 \qquad n \ge n_0.
 \label{JN1}
 \tag{$\mathfrak{D}_1$}
 \end{equation}
\end{thm}

The condition
 \[
 \limsup_{n\to\infty}
 \left(\frac{a_n^2}{b_{n+1}^2}+\frac{a_{n-1}^2}{b_{n-1}^2}\right) < \frac{1}{2},
 \]
is slightly stronger than \eqref{JN1}, hence it also implies the self-adjointness of $T$. This condition is similar but different from another one due to J. Janas and S. Naboko, namely,
 \begin{equation}
 \limsup_{n\to\infty} \frac{a_n^2+a_{n-1}^2}{b_n^2} < \frac{1}{2}.
 \label{JN}
 \end{equation}
This, together with the divergence of $|b_n|$, guarantees that $T$ is self-adjoint with a discrete spectrum \cite{JN2001}.

In the case $\lim_{n\to\infty}a_n=\infty$, Janas-Naboko condition \eqref{JN} can be understood as a special case of a more general condition for self-adjointness and discreteness of the spectrum developed by P. Cojuhari and J. Janas in \cite{CJ2007}. In this work the authors study symmetric tridiagonal operators defined by $a_n=-\alpha_n$ and $b_n=\alpha_{n-1}+\alpha_n+\beta_n$ with $\alpha_n>0$ and $\beta_n\ge0$ for big enough $n$ (indeed they deal with generalizations of these operators to weighted $\ell^2$ spaces). The change of basis $e_n \to (-1)^n e_n$ shows that one can set $a_n=\alpha_n$ without modifying the expression for $b_n$, thus fitting with our choice $a_n>0$ for $T$ as stated in \eqref{Intro_T}. Bearing this in mind, the result of interest for us \cite[Theorem 3.2 (i)]{CJ2007} can be rewritten by saying that $T$ is self-adjoint with a discrete spectrum if
 \begin{equation}
 \begin{gathered}
 \beta_n=b_n-a_n-a_{n-1}>0 \; \text{ for big enough } n,
 \\
 \lim_{n\to\infty}a_n=\infty, \qquad
 \lim_{n\to\infty}(\beta_n+\beta_{n+1})=\infty.
 \end{gathered}
 \label{CJ}
 \end{equation}

Although the arguments given previously do not ensure the discreteness of the spectrum under \eqref{JN1}, in contrast to \eqref{JN} or \eqref{CJ}, these arguments do not require the divergence of $|b_n|$ neither the inequality $b_n \ge a_n+a_{n-1}$.  Furthermore, they have the advantage of being generalizable to yield infinitely many conditions for self-adjointness (concrete comparisons between these infinitely many conditions and \eqref{JN} or \eqref{CJ} will be shown in \S\ref{Comparisons}).

For instance, inserting the $n-1$-th and $n+1$-th inequalities of \eqref{ineq_vn_bis} into the $n$-th one we get
{\scriptsize
 \[
 |v_n|^2 \le
 4 \left\{
 (\gamma_n^-)^2
 \left[(\gamma_{n-1}^-)^2 |v_{n-2}|^2 + (\gamma_{n-1}^+)^2 |v_n|^2\right]
 + (\gamma_n^+)^2
 \left[(\gamma_{n+1}^-)^2|v_n|^2+(\gamma_{n+1}^+)^2|v_{n+2}|^2\right]
 \right\}.
 \]
}
This implies that
{\scriptsize
 \[
 \sum_{n=1}^\infty |v_n|^2 \le
 4 \sum_{n=1}^\infty
 \left[
 (\gamma_{n+2}^-)^2 (\gamma_{n+1}^-)^2 + (\gamma_n^-)^2 (\gamma_{n-1}^+)^2
 + (\gamma_n^+)^2 (\gamma_{n+1}^-)^2 + (\gamma_{n-2}^+)^2 (\gamma_{n-1}^+)^2
 \right]
 |v_n|^2.
 \]
}
In consequence, $T$ must be self-adjoint under the condition
 \[
 (\gamma_{n+2}^-)^2 (\gamma_{n+1}^-)^2 + (\gamma_n^-)^2 (\gamma_{n-1}^+)^2
 + (\gamma_n^+)^2 (\gamma_{n+1}^-)^2 + (\gamma_{n-2}^+)^2 (\gamma_{n-1}^+)^2
 < \frac{1}{4},
 \quad n\ge1.
 \]
Using finite rank perturbations, just as in the previous case, this result leads to the following more general one.

\begin{thm} \label{thm-JN2}
$T$ is self-adjoint if there exists an index $n_0\in\mathbb{N}$ such that
 \begin{equation}
 \frac{a_n^2}{b_{n+1}^2}
 \left(\frac{a_{n+1}^2}{b_{n+2}^2} + \frac{a_n^2}{b_n^2}\right)
 + \frac{a_{n-1}^2}{b_{n-1}^2}
 \left(\frac{a_{n-1}^2}{b_n^2} + \frac{a_{n-2}^2}{b_{n-2}^2}\right)
 < \frac{1}{4},
 \qquad n \ge n_0.
 \label{JN2}
 \tag{$\mathfrak{D}_2$}
 \end{equation}
\end{thm}

\eqref{JN1} and \eqref{JN2} are again particular cases of general $m$-conditions for self-adjointness. They are obtained by an iterative use of the eigenvalue equations \eqref{Intro_OPRR} via the inequality \eqref{ineq_vn_bis}.

\begin{thm} \label{thm-JNm}
For any $m\in\mathbb{N}$, the existence of an index $n_0\in\mathbb{N}$ such that
 \begin{equation}
 \widetilde{G}_{m,n} < \frac{1}{2^m},
 \kern9pt
 \widetilde{G}_{m,n} = \kern-12pt
 \sum_{(\boldsymbol{j}_{m+1}|\boldsymbol{k}_m)\in\widehat{\mathcal{I}}_m}
 \frac{a_{n+j_{m+1}-k_1}^2}{b_{n+j_{m+1}-j_1}^2}
 \cdots \frac{a_{n+j_{m+1}-k_m}^2}{b_{n+j_{m+1}-j_m}^2},
 \kern9pt n\ge n_0,
 \label{JNm}
 \tag{$\mathfrak{D}_m$}
 \end{equation}
implies that $T$ is self-adjoint.
\end{thm}

\begin{proof}
Assume that $b_n\ne0$ for $n\ge1$. Proceeding by induction analogously to the proof of Proposition \ref{deltaN}, we find from \eqref{ineq_vn_bis} that
 \[
 |v_n|^2 \le 2^m
 \sum_{(\boldsymbol{j}_{m+1}|\boldsymbol{k}_m)\in\widehat{\mathcal{I}}_m}
 \frac{a_{n-k_1}^2}{b_{n-j_1}^2} \frac{a_{n-k_2}^2}{b_{n-j_2}^2}
 \cdots \frac{a_{n-k_m}^2}{b_{n-j_m}^2} \, |v_{n-j_{m+1}}|^2,
 \]
where $a_k=v_k=0$ for $k\le0$. Summing for $n\ge1$ we obtain
 \[
 \sum_{n=1}^\infty |v_n|^2 \le 2^m \sum_{n=1}^\infty
 \widetilde{G}_{m,n}
 \, |v_n|^2.
 \]
Therefore, $T$ is self-adjoint whenever $\widetilde{G}_{m,n} < 2^{-m}$. The theorem follows from this result resorting to finite rank perturbations, just as in the case of Theorem \ref{thm-JN1}.
\end{proof}

A weaker but more practical version of this theorem reads as follows.

\begin{cor} \label{cor-JNm}
For any $m\in\mathbb{N}$, the condition
 \[
 \limsup_{n\to\infty} \widetilde{G}_{m,n} < \frac{1}{2^m}
 \]
implies that $T$ is self-adjoint.
\end{cor}

The following consequence of Theorem \ref{thm-JNm} should be compared with Theorem \ref{strong} and Corollary \ref{cor-Car}.

\begin{thm} \label{cor-JN}
For any $m\in\mathbb{N}$, the condition
 \[
 \lim_{n\to\infty} G_{m,n} = 0
 \]
implies that $T$ is self-adjoint.
\end{thm}

\begin{proof}
The result follows from Theorem \ref{thm-JNm} and the equivalences
 \[
 \begin{aligned}
 & \lim_{n\to\infty} \widetilde{G}_{m,n} = 0
 \;\Leftrightarrow\;
 \lim_{n\to\infty}
 \frac{a_{n+j_{m+1}-k_1}}{b_{n+j_{m+1}-j_1}}
 \cdots
 \frac{a_{n+j_{m+1}-k_m}}{b_{n+j_{m+1}-j_m}}
 = 0,
 \kern5pt \forall (\boldsymbol{j}_{m+1}|\boldsymbol{k}_m) \in \widehat{\mathcal{I}}_m,
 \\
 & \;\Leftrightarrow\;
 \lim_{n\to\infty}
 \frac{a_{n-k_1}}{b_{n-j_1}}
 \cdots
 \frac{a_{n-k_m}}{b_{n-j_m}}
 = 0,
 \kern5pt \forall (\boldsymbol{j}_m|\boldsymbol{k}_m) \in \mathcal{I}_m
 \;\Leftrightarrow\;
 \lim_{n\to\infty} G_{m,n} = 0.
 \end{aligned}
 \]
\end{proof}

 \section{Examples and comparisons of $m$-conditions}
 \label{Comparisons}
 \setcounter{equation}{0}
 \renewcommand{\theequation}{\thesection.\arabic{equation}}
 \setcounter{thm}{0}
 \renewcommand{\thethm}{\thesection.\arabic{thm}}
 \par

We will compare the previous sets of $m$-conditions with known results for self-adjointness. Before doing this we will discuss the relation between $m$-conditions for different values of $m$ to understand the relevance of developing sets of infinitely many different conditions for self-adjointness.

In what follows we use the common notations $y_n \sim z_n$ and $y_n \asymp z_n$, which stand for the relations $\lim_{n\to\infty}y_n/z_n=1$ and $C_1z_n \le |y_n| \le C_2z_n$ ($C_1,C_2>0$ and $n$ big enough) respectively.

As we mentioned in \S\ref{BR}, conditions \eqref{eq-strong} are all
independent, so that all of them are equally important. This is shown by the following example.

\begin{exa} \label{ex-B} ${}$
\begin{enumerate}

\item
$a_n=n^\alpha$, $b_n=n^{\alpha+1}$, with $\alpha>0$.

This choice leads to $a_nG_{m,n}^+ \asymp n^{\alpha-m}$, so $\lim_{n\to\infty}a_nG_{m,n}^+\ne0$ for $m\le\alpha$, while $\lim_{n\to\infty}a_nG_{m,n}^+=0$ for $m>\alpha$.

\item $a_n=n^\alpha$ for even $n$, $a_n=n^{-\alpha}$ for odd $n$, $b_n=n^{\alpha-1}$, with $\alpha>1$.

In this case,
 \[
 \begin{aligned}
 & a_nG_{m,n}^+ \sim a_n
 \left(
 \frac{a_{n-1}}{b_n} \frac{a_{n-1}}{b_{n-1}}
 \frac{a_{n-1}}{b_n} \frac{a_{n-1}}{b_{n-1}}
 \cdots
 \right)
 & & \text{ odd } n,
 \\
 & a_nG_{m,n}^+ \sim a_n \frac{a_{n-1}}{b_n}
 \left(
 \frac{a_{n-2}}{b_{n-1}} \frac{a_{n-2}}{b_{n-2}}
 \frac{a_{n-2}}{b_{n-1}} \frac{a_{n-2}}{b_{n-2}}
 \cdots
 \right)
 & & \text{ even } n.
 \end{aligned}
 \]
Therefore, $a_nG_{m,n}^+ \sim n^{m-\alpha}$, which implies that $\lim_{n\to\infty}a_nG_{m,n}^+\ne0$ for $m\ge\alpha$ and $\lim_{n\to\infty}a_nG_{m,n}^+=0$ for $m<\alpha$.
\end{enumerate}

In both examples we conclude that $T$ is self-adjoint and $\Lambda(T)=\sigma(T)$, but the $m$-conditions providing these results are different. Besides, given $m_0\in\mathbb{N}$, \eqref{eq-strong} is satisfied for $m=m_0$ but not for $m<m_0$ in the first example with $\alpha=m_0-1$, while it is satisfied for $m=m_0$ but not for $m>m_0$ in the second example with $\alpha=m_0+1$. This shows that given two of the conditions \eqref{eq-strong}, none of them is stronger than the other one.
\end{exa}

The next example illustrates a similar complementarity for the Carleman type conditions \eqref{eq-Carm}.

\begin{exa} \label{ex-C} ${}$
\begin{enumerate}

\item Example \ref{ex-B}.1.

We get $a_nG_{m,n} \asymp n^{\alpha-m}$, thus $\sum_{n=m+1}^\infty 1/a_nG_{m,n} < \infty$ for $m<\alpha-1$ and $\sum_{n=m+1}^\infty 1/a_nG_{m,n} = \infty$ for $m\ge\alpha-1$.

\item $a_n=n^{1/\alpha}$, $b_n=1$, with $\alpha\ge1$.

We find that $a_nG_{m,n} \asymp n^{(m+1)/\alpha}$. Thus, $\sum_{n=m+1}^\infty 1/a_nG_{m,n} < \infty$ for $m>\alpha-1$ and $\sum_{n=m+1}^\infty a_nG_{m,n} = \infty$ for $m\le\alpha-1$.

\end{enumerate}

We find again that $T$ is self-adjoint in both cases. However, if $\alpha=m_0+1$, \eqref{eq-Carm} holds for $m=m_0$ and not for $m<m_0$ in the first example, while it holds for $m=m_0$ and not for $m>m_0$ in the second example. This shows the independence of conditions \eqref{eq-Carm}.
\end{exa}

Regarding conditions \eqref{JNm} the situation is somewhat different. To see this let us use the definitions of $\widehat{\mathcal{I}}_m$ and $\widehat{\mathcal{I}}_{m+1}$ to write
 \begin{equation} \label{rec-G}
 \begin{aligned}
 \widetilde{G}_{m+1,n} & =
 \sum_{(\boldsymbol{j}_{m+1}|\boldsymbol{k}_m)\in\widehat{\mathcal{I}}_m}
 \frac{a_{n+j_{m+1}-1-k_1}^2}{b_{n+j_{m+1}-1-j_1}^2}
 \cdots \frac{a_{n+j_{m+1}-1-k_m}^2}{b_{n+j_{m+1}-1-j_m}^2}
 \frac{a_{n-1}^2}{b_{n-1}^2}
 \\
 & + \sum_{(\boldsymbol{j}_{m+1}|\boldsymbol{k}_m)\in\widehat{\mathcal{I}}_m}
 \frac{a_{n+j_{m+1}+1-k_1}^2}{b_{n+j_{m+1}+1-j_1}^2}
 \cdots \frac{a_{n+j_{m+1}+1-k_m}^2}{b_{n+j_{m+1}+1-j_m}^2}
 \frac{a_n^2}{b_{n+1}^2}
 \\
 & = \frac{a_{n-1}^2}{b_{n-1}^2} \, \widetilde{G}_{m,n-1}
 + \frac{a_n^2}{b_{n+1}^2} \, \widetilde{G}_{m,n+1}.
 \end{aligned}
 \end{equation}
Hence, $\widetilde{G}_{m+1,n} \le \widetilde{G}_{1,n} \max\{\widetilde{G}_{m,n-1},\widetilde{G}_{m,n+1}\}$, which shows by induction that \eqref{JN1} implies the rest of conditions \eqref{JNm}. The interest in conditions with higher values of $m$ rests on the existence of examples satisfying \eqref{JNm} for a given value of $m$, but not for any smaller index. This is illustrated by the next example.

\begin{exa} \label{ex-D}
$a_n=a_{n-1}$ if $n=0\,(\operatorname{mod} q)$, $a_n=n^{q+1}a_{n-1}$ otherwise, $b_n=n^qa_{n-1}$, with $q\in\{2,3,\dots\}$ and $a_1=b_1=1$.

In this case $a_n/b_n=n^{-q}$ if $n=0\,(\operatorname{mod} q)$, $a_n/b_n=n$ otherwise and $a_{n-1}/b_n=n^{-q}$. We find from \eqref{rec-G} that
 \[
 \widetilde{G}_{m,n} \ge
 \frac{a_{n-1}^2}{b_{n-1}^2} \frac{a_{n-2}^2}{b_{n-2}^2} \cdots
 \frac{a_{n-m}^2}{b_{n-m}^2} \sim n^{2m},
 \qquad n=0\,(\operatorname{mod} q), \qquad m<q,
 \]
so that \eqref{JNm} does not hold for this example if $m<q$. This changes when $m=q$ because any term of $\widetilde{G}_{m=q,n}$ has among its factors at least one with the form $a_k^2/b_k^2$, $k=0\,(\operatorname{mod} q)$, or with the form $a_{k-1}^2/b_k^2$. Thus, $\widetilde{G}_{m=q,n}\asymp n^{-2q} n^{2(m-1)}=n^{-2}$ and the $m$-condition is satisfied for $m=q$, proving that $T$ is self-adjoint.
\end{exa}

Let us see now that the $m$-conditions allow us to prove the self-adjointness of certain examples where known results give no or less information. We will see that this is the case of known self-adjointness criteria given in terms of the coefficients $a_n$, $b_n$, like Carleman \eqref{CAR}, Dennis-Wall \eqref{eq-DW}, Janas-Naboko \eqref{JN} (for the case of divergent $|b_n|$) or Cojuhari-Janas \eqref{CJ} conditions.

A first instance of this is given by Example \ref{ex-D}. Then, $(a_n^2+a_{n-1}^2)/b_n^2 \sim n^2$ and $\beta_n=b_n-a_n-a_{n-1}=(n^q-n^{q+1}-1)a_{n-1}<0$ for $n\ne0\,(\operatorname{mod} q)$, so neither \eqref{JN} nor \eqref{CJ} are satisfied. Besides, the relation between $a_n$ and $a_{n-1}$ shows that Carleman's condition does not hold either. This behavior of $a_n$ also proves that Dennis-Wall condition is not applicable since $b_n/a_na_{n-1}=n^q/a_n$. Therefore, Example \ref{ex-D} shows that the $m$-conditions \eqref{JNm} can improve the results obtained using conditions \eqref{CAR}, \eqref{eq-DW}, \eqref{JN} and \eqref{CJ}. The next examples illustrate this fact regarding the $m$-conditions \eqref{eq-strong} and \eqref{eq-Carm}, which give no information in Example \ref{ex-D}.

\begin{exa} \label{ex-B-comp}
$a_n=n^\alpha$, $b_n=n^\beta$ for even $n$, $b_n=n^\gamma$ for odd $n$, with $\alpha,\beta,\gamma>1$.

Concerning condition \eqref{JN}, $(a_n^2+a_{n-1}^2)/b_n^2 \sim 2n^{2(\alpha-\beta)}$ for even $n$, while $(a_n^2+a_{n-1}^2)/b_n^2 \sim 2n^{2(\alpha-\gamma)}$ for odd $n$. Hence, Janas-Naboko criterion guarantees the self-adjointness of $T$ for $\alpha<\min\{\beta,\gamma\}$. The same result is obtained from Cojuhari-Janas condition \eqref{CJ} because $\beta_n=n^\beta-n^\alpha-(n-1)^\alpha$ or $\beta_n=n^\gamma-n^\alpha-(n-1)^\alpha$ depending whether $n$ is even or odd. Since $\sum_{n=1}^\infty n^{-\alpha} < \infty$ for $\alpha>1$, Carleman's criterion is not applicable.
On the other hand, $|b_n|/a_na_{n-1} \sim n^{\beta-2\alpha}$ for even $n$ and  $|b_n|/a_na_{n-1} \sim n^{\gamma-2\alpha}$ for odd $n$, thus Dennis-Wall criterion states that $T$ is self-adjoint when $\alpha\le(\max\{\beta,\gamma\}+1)/2$.

To compare these results with that one provided by Theorem \ref{strong}, note that
 \[
 \begin{aligned}
 & a_nG_{m,n}^+ \asymp n^{(m+1)\alpha-k(\beta+\gamma)},
 & & \quad \text{even } m=2k,
 \\
 & a_nG_{m,n}^+ \asymp n^{(m+1)\alpha-(k+1)\beta-k\gamma},
 & & \quad \text{odd } m=2k+1, \; \text{ even } n,
 \\
 & a_nG_{m,n}^+ \asymp n^{(m+1)\alpha-k\beta-(k+1)\gamma},
 & & \quad \text{odd } m=2k+1, \; \text{ odd } n.
 \end{aligned}
 \]
Therefore, the requirement $\lim_{n\to\infty} a_nG_{m,n}^+=0$ reads as
 \[
 \begin{aligned}
 & 2(m+1)\alpha < m(\beta+\gamma),
 & & \quad \text{even } m,
 \\
 & 2(m+1)\alpha < (m+1)(\beta+\gamma)-2\max\{\beta,\gamma\},
 & & \quad \text{odd } m.
 \end{aligned}
 \]
We conclude that condition \eqref{eq-strong} holds for some value of $m$ iff the parameters $\alpha,\beta,\gamma$ satisfy the $m\to\infty$ inequality $\alpha < (\beta+\gamma)/2$. This improves the results given by the four previous criteria, not only because the self-adjointness of $T$ is ensured for a bigger region in the space of parameters $\alpha,\beta,\gamma>1$, but also because we get additionally the equality $\Lambda(T)=\sigma(T)$ in that region.
\end{exa}

Theorems \ref{Carm} and \ref{thm-JNm} give no additional information for the previous example. However, this fact changes in the following one.

\begin{exa} \label{ex-C-comp}
$a_n=n^\alpha$, $b_n=b^n$ if $n\in\Delta$, $b_n=n^\beta$ otherwise, with $\alpha,\beta>1$, $0<b<1$ and $\Delta=\{k^2:k\in\mathbb{N}\}$.

Conditions \eqref{JN} and \eqref{CJ} do not hold because $(a_n^2+a_{n-1}^2)/b_n^2 \asymp n^{2\alpha}b^{-2n}$ and $\beta_n=b^n-n^\alpha-(n-1)^\alpha<0$ for $n\in\Delta$. Carleman's criterion gives no information since $\alpha>1$. As for Dennis-Wall criterion, $|b_n|/a_na_{n-1} \sim b^nn^{-2\alpha}$ for $n\in\Delta$, while $|b_n|/a_na_{n-1} \sim n^{\beta-2\alpha}$ for $n\notin\Delta$. Thus, the divergence of $\sum_{n\ge1}|b_n|/a_na_{n-1}$ is equivalent to the divergence of $\sum_{n\notin\Delta} n^{-\gamma}$, $\gamma=2\alpha-\beta$, which diverges simultaneously with $\sum_{n\ge1} n^{-\gamma}$. This last statement, obvious when $\gamma\le0$, follows in the case $\gamma>0$ from the decreasing character of $n^{-\gamma}$ which leads to the inequality $\sum_{n\in\Delta} n^{-\gamma} \le 1 + \sum_{n\notin\Delta} n^{-\gamma}$. This proves that $\sum_{n\notin\Delta} n^{-\gamma} \le \sum_{n\ge1} n^{-\gamma} \le 1 + 2\sum_{n\notin\Delta} n^{-\gamma}$ for $\gamma>0$. Therefore, Dennis-Wall criterion guarantees the self-adjointness of $T$ for $\gamma\le1$, i.e. $\alpha \le (\beta+1)/2$.

Concerning the new self-adjointness $m$-conditions given in this paper, $|b_n|$ is unbounded but not divergent, which prevents the use of Theorem \ref{strong}. Regarding conditions \eqref{JNm}, none of them hold due to the inequalities
 \[
 \widetilde{G}_{m,n} \ge
 \frac{a_{n-1}^2}{b_{n-1}^2} \frac{a_{n-2}^2}{b_{n-2}^2}
 \cdots \frac{a_{n-m}^2}{b_{n-m}^2}
 \ge b^{-2n} n^{2[m\alpha-(m-1)\beta]},
 \qquad n-1\in\Delta.
 \]

To analyze the Carleman type conditions \eqref{eq-Carm}, let us consider the subsets $\Delta_m = \{n\in\mathbb{N}:\operatorname{dist}(n,\Delta)<m\}$. If $n\in\mathbb{N}\setminus\Delta_m$, then $b_{n-j}\sim n^\beta$ for $|j|<m$, hence $a_nG_{m,n} \asymp n^{(m+1)\alpha-m\beta}$. Since $\sum_{n\ge m+1} 1/a_nG_{m,n} \ge \sum_{n\in\mathbb{N}\setminus\Delta_m} 1/a_nG_{m,n}$, the divergence of $\sum_{n\in\mathbb{N}\setminus\Delta_m} n^{-\delta}$ for some of the coefficients $\delta=(m+1)\alpha-m\beta$ implies the self-adjointness of $T$.

Let us prove that $\sum_{n\in\mathbb{N}\setminus\Delta_m} n^{-\delta}$ and $\sum_{n\in\mathbb{N}} n^{-\delta}$ diverge simultaneously, i.e. when $\delta\le1$. Obviously $\sum_{n\in\mathbb{N}\setminus\Delta_m} n^{-\delta} = \infty$ for $\delta\le0$. Let us suppose that $\delta>0$. Since $\lim_{k\to\infty}[(k+1)^2-k^2]=\infty$, for any $m\in\mathbb{N}$ there exists $k_0\in\mathbb{N}$ such that $\mathbb{N}\setminus\Delta_m$ has at least $2m-1$ points in $[(k-1)^2,k^2]$ for $k\ge k_0$. Then, the decreasing character of $n^{-\delta}$ for $\delta>0$ ensures that
 \[
 \sum_{n\in(\mathbb{N}\setminus\Delta_m)\cap[(k-1)^2,k^2]}
 \kern-25pt n^{-\delta}
 \kern15pt \ge
 \sum_{n\in\mathbb{N}\cap(k^2-m,k^2+m)} \kern-17pt n^{-\delta},
 \qquad k\ge k_0,
 \]
which gives
 \[
 \sum_{n\in\mathbb{N}\setminus\Delta_m} n^{-\delta} \kern7pt
 \ge \sum_{n\in\Delta_m\cap(k_0^2-m,\infty)} \kern-15pt n^{-\delta}.
 \]
Therefore,
 \[
 \sum_{n\in\mathbb{N}\setminus\Delta_m} n^{-\delta} \kern7pt
 \le \kern7pt \sum_{n\in\mathbb{N}} \kern5pt n^{-\delta} \kern5pt
 \le \sum_{n\in\Delta_m\cap[1,k_0^2-m]} \kern-15pt n^{-\delta} \kern3pt
 + \kern3pt 2 \kern-3pt \sum_{n\in\mathbb{N}\setminus\Delta_m} n^{-\delta},
 \]
proving that the divergence of $\sum_{n\in\mathbb{N}\setminus\Delta_m} n^{-\delta}$ is equivalent to the divergence of $\sum_{n\in\mathbb{N}} n^{-\delta}$ for $\delta>0$ too.

We conclude that, as a consequence of the Carleman type $m$-conditions, $T$ is self-adjoint whenever $(m+1)\alpha-m\beta\le1$ for some $m\in\mathbb{N}$, i.e. when $\alpha<\beta$. This result improves the one obtained with Dennis-Wall criterion.
\end{exa}

The application of Carleman type $m$-conditions to the above example does not depend on the precise values of $b_n$ for $n\in \Delta$, neither on the details of the set $\Delta=\{n_1,n_2,\dots\}$ provided that $\lim_{k\to\infty}(n_{k+1}-n_k)=\infty$. Therefore, the results of Example \ref{ex-C-comp} are the same assuming only this general property of the set $\Delta$, for any choice of $b_n$ on this set.

Moreover, the arguments in the above example apply, {\it mutatis mutandis}, to give the following general result: if the coefficients $a_n$, $b_n$ satisfy \eqref{eq-Carm} for some value of $m$ and give a non decreasing sequence $a_nG_{m,n}$ for $n$ big enough, then not only $T$ is self-adjoint, but also any other symmetric tridiagonal operator obtained from $T$ by perturbing arbitrarily the coefficients $a_n$, $b_n$ on a set $\Delta=\{n_1,n_2,\dots\}$ with $\lim_{k\to\infty}(n_{k+1}-n_k)=\infty$.

It is worth remarking that in all the previous examples the conclusions remain unchanged when substituting the equality in the choice of $a_n$ and $b_n$ by the asymptotic condition $\asymp$. For instance, similar arguments to those given in Example~\ref{ex-B-comp} prove that $T$ is self-adjoint and $\Lambda(T)=\sigma(T)$ if $a_n\asymp n^\alpha$, $b_n\asymp n^\beta$ for even $n$, $b_n\asymp n^\gamma$ for odd $n$, with $\beta,\gamma>0$ and $\alpha<(\beta+\gamma)/2$.

 \section{Applications}
 \label{Applications}
 \setcounter{equation}{0}
 \renewcommand{\theequation}{\thesection.\arabic{equation}}
 \setcounter{thm}{0}
 \renewcommand{\thethm}{\thesection.\arabic{thm}}
 \par

As we mentioned in the introduction, when $T$ is self-adjoint, the related Jacobi continued fraction $K(\lambda)$ given in \eqref{Intro_CF} converges to the first diagonal element $((\lambda-T)^{-1}e_1,e_1)$ of the resolvent of $T$ for every $\lambda\in\mathbb{C}\setminus\Lambda(T)$. Moreover, this convergence is uniform on compact subsets of $\mathbb{C}\setminus\Lambda(T)$. From this point of view, any information about the set $\Lambda(T)$ is of interest because it gives also information about the analyticity properties of $K(\lambda)$.

On the other hand, Theorem \ref{strong} establishes conditions under which the knowledge of $\Lambda(T)$ is equivalent to the knowledge of the more accessible set given by the spectrum $\sigma(T)$ of $T$. This permits us to apply techniques of spectral theory to study the analyticity properties of Jacobi continued fractions. In this regard, the special case of Theorem \ref{strong} given by Theorem \ref{weak} is particularly useful due to the simplicity of its hypothesis which make them easily verifiable.

Theorem \ref{weak} becomes also especially interesting due to its consequences concerning the properties of $\sigma(T)$. It is known that the divergence of $|b_n|$ and condition \eqref{JN} imply that $T$ has a pure point spectrum $\sigma(T)=\cup_n\{\lambda_n\}$ with $|\lambda_n|$ divergent \cite{JN2001}. Since $a_n,a_{n-1}=o(|b_n|)$ implies \eqref{JN}, we find that the hypothesis of Theorem \ref{weak} ensure this kind of unbounded discrete spectrum.

This has remarkable consequences for the Jacobi continued fraction $K(\lambda)$. If $T$ is self-adjoint and $\Lambda(T)=\sigma(T)$, then $K(\lambda)$ represents a meromorphic function in $\mathbb{C}$ precisely when $T$ has a discrete spectrum with eigenvalues $\lambda_n$ such that $\lim_{n\to\infty}|\lambda_n| = \infty$ \cite{IKP2007}. Therefore, Theorem \ref{weak} has the following implications for the convergence of Jacobi continued fractions.

\begin{thm} \label{convergence}
The conditions
 \[
 \begin{aligned}
 & \lim_{n\to\infty}|b_n|=\infty,
 \\
 & a_n,a_{n-1}=o(|b_n|^r) \text{ for some } r<1,
 \end{aligned}
 \]
imply that the Jacobi continued fraction $K(\lambda)$ given in \eqref{Intro_CF} represents a meromorphic function in $\mathbb{C}$. This continued fraction converges uniformly on compact subsets of $\mathbb{C}\setminus\cup_n\{\lambda_n\}$, where $\lambda_n$ are the eigenvalues of $T$.
\end{thm}

The relevance of Theorem \ref{weak} is also illustrated by the examples in the literature which are covered by this theorem. For instance, this is the case of \cite{JM2007,JN2004,M2007,M2009,M2010,V2004}, which deal with the asymptotic analysis of the eigenvalues of $T$ for different choices of coefficients with a power like behaviour $a_n \asymp n^\alpha$, $b_n \asymp n^\beta$ (more generally, $a_n \le an^\alpha$, $b_n \ge bn^\beta$), where $\alpha<\beta$. Theorem \ref{weak} provides a computational method to approximate such eigenvalues and suggests an approach to their asymptotics by studying the eigenvalues of the truncated operators $T_N$ as $N\to\infty$, a technique already exploited for example in \cite{A1994,IKP2007,IP2001,M2007,M2010,MZ2012,V2004}.

 \subsubsection*{Acknowledgements}
 E. Petropoulou would like to express her gratitude to the Department of Applied Mathematics of the University of Zaragoza, Spain, for the hospitality during her sabbatical leave there, from February 2012 until July 2012, when this work was initiated.

 The research of L. Vel\'azquez is partially supported by the research project MTM2011-28952-C02-01 from the Ministry of Science and Innovation of Spain and the European Regional Development Fund (ERDF), and by Project E-64 of Diputaci\'on General de Arag\'on (Spain).

 \end{document}